\documentclass[final]{siamltex}
\usepackage{amssymb,epsfig,amsmath,multirow}

\newtheorem{exam}{Example}[section]

\newtheorem{rem}{Remark}[section]

\title{An SDP Approach For Solving Quadratic Fractional
Programming Problems
\thanks{This research was supported by Taiwan National Science Council under
grant 102-2115-M-006-010, by National Center for Theoretical
Sciences (South), by National Natural Science Foundation of China
under grants 11001006 and 91130019/A011702, and by the fund of State
Key Laboratory of Software Development Environment under grant
SKLSDE-2013ZX-13} }

\author{Van-Bong Nguyen\footnotemark[2]  \and Ruey-Lin Sheu\footnotemark[3]~\footnotemark[4]
\and Yong  Xia\footnotemark[5]}

\begin{document}
\maketitle
\renewcommand{\thefootnote}{\fnsymbol{footnote}}

\footnotetext[2]{Department of Mathematics, National Cheng Kung
University, Taiwan ({\tt
bongtnmath@yahoo.com.vn}).}

\footnotetext[3]{Department of Mathematics, National Cheng Kung
University, Taiwan ({\tt
rsheu@mail.ncku.edu.tw}).}

\footnotetext[4]{Corresponding author.}

\footnotetext[5]{State Key Laboratory of Software Development
              Environment, LMIB of the Ministry of Education, School of
Mathematics and System Sciences, Beihang University, Beijing,
100191, P. R. China ({\tt dearyxia@gmail.com}).}

\begin{abstract}
This paper considers a fractional programming problem (P) which
minimizes a ratio of quadratic functions subject to a two-sided
quadratic constraint. As is well-known, the fractional objective
function can be replaced by a parametric family of quadratic
functions, which makes (P) highly related to, but more difficult
than a single quadratic programming problem subject to a similar
constraint set. The task is to find the optimal parameter
$\lambda^*$ and then look for the optimal solution if $\lambda^*$ is
attained. Contrasted with the classical Dinkelbach method that
iterates over the parameter, we propose a suitable constraint
qualification under which a new version of the S-lemma with an
equality can be proved so as to compute $\lambda^*$ directly via an
exact SDP relaxation. When the constraint set of (P) is degenerated
to become an one-sided inequality, the same SDP approach can be
applied to solve (P) {\it without any condition}. We observe that
the difference between a two-sided problem and an one-sided problem
lies in the fact that the S-lemma with an equality does not have a
natural Slater point to hold, which makes the former essentially
more difficult than the latter. This work does not, either, assume
the existence of a positive-definite linear combination of the
quadratic terms (also known as the dual Slater condition, or a
positive-definite matrix pencil), our result thus provides a novel
extension to the so-called ``hard case'' of the generalized trust
region subproblem subject to the upper and the lower level set of a
quadratic function.
\end{abstract}


\pagestyle{myheadings}
\thispagestyle{plain}
\markboth{V.B. Nguyen, R.L. Sheu and Y. Xia}{An SDP approach for solving quadratic fractional
programming problems}

\section{Introduction}
\label{sec:1}
In this paper we study a single ratio quadratic fractional programming problem
taking the following format:
\begin{equation}\label{P}
{\rm (P)}~\lambda^*=\inf \left\{\frac{f_1(x)}{f_2(x)} : x\in X \right\}
\end{equation}
where $X=\{x\in \Bbb R^n: u\le g(x)\le v\};$
$f_i(x)=x^TA_ix+2b_i^Tx+c_i,\ i=1,2;\  g(x)=x^TBx+2d^Tx+\alpha.$ The
matrices  $A_1, A_2, B$ are assumed to be symmetric and can be
indefinite, $u\in \Bbb R\cup\{-\infty\}, v\in \Bbb R\cup\{+\infty\}$
such that $X\ne\emptyset.$ To avoid the denominator becoming 0, we
call (P) {\it well-defined} if $f_2(x)>0$ for all $x\in X.$ In this
paper, we only consider a well-defined (P), but characterize
conditions under which (P) can be well-defined in the last section.
Denote $x^*$ to be the optimal solution of (P) if it is attained,
and $\lambda^*$ the infimum of the problem, which could be $-\infty$
when (P) is unbounded below. By setting $f_2(x)=1$, problem (P) is
reduced to the ``interval bounded generalized trust region
subproblem (I-GTRS)'' \cite{Pong} which is essentially a quadratic
programming problem with two quadratic constraints (QP2QC). Problem
(I-GTRS) was studied in \cite{SW} by Stern and Wolkowicz for a
homogenous $g(x)$; in \cite{Y} by Ye and Zhang under a primal and
dual Slater condition; and in \cite{Pong} by Pong and Wolkowicz for
a necessary and sufficient optimality condition with an algorithm
solving the ``regular case'' (to be explained later). Due to the
fractional structure in the objective, (P) is in general more
difficult than (I-GTRS).

As is well-known, the fractional objective function can be replaced
by a parametric family of quadratic functions. Dinkelbach
\cite{Dink} in 1967 proposed a family of subproblems
parameterized by $\lambda:$
\begin{equation}\label{P_lambda}
{\rm (P)}_\lambda\hspace*{1cm}f(\lambda)=\inf \left\{ f_1(x)-\lambda
f_2(x) : x\in X \right\}
\end{equation}
and developed an iterative algorithm on $\lambda$ to find a value
$\lambda_0$ such that $f(\lambda_0)=0.$ When $X$ is compact, it was
shown that $\lambda_0=\lambda^*$. Moreover, (P) and
(P)$_{\lambda_0}$ share the same optimal solution set
\cite{Dink,Iba,ALi}. Applying the Dinkelbach method to solve (P)
amounts to solving {\it globally} a sequence of (I-GTRS)'s. Each
(I-GTRS) $(P)_\lambda$ could be unbounded below or unattainable.
Otherwise, under the primal Slater condition:\\
\ \\
{\bf Assumption A}\hspace*{1.5cm} $\inf_{x\in \Bbb R^n} g(x)<u\le
v<\sup_{x\in \Bbb R^n} g(x),
$\\
\ \\
a global optimal solution $x(\lambda)$ to $(P)_\lambda$ can be
characterized with a Lagrange multiplier $\mu(\lambda)$ such that
the first order condition $(A_1-\lambda
A_2-\mu(\lambda)B)x(\lambda)=-b_1+\lambda b_2 +\mu(\lambda)d$; the
second order condition $A_1-\lambda A_2-\mu(\lambda)B\succeq0$;
together with the complementarity become necessary and sufficient
\cite{Pong}. The real task is to find algorithmically the pair of
saddle point $(x(\lambda),\mu(\lambda))$ for each $\lambda$, suppose
they exist, from the set of optimality conditions. So far, existing
methods such as SDP with a rank one decomposition procedure \cite{Y}
or a matrix pencil secular function approach \cite{Pong} must rely
on the existence of a positive definite matrix pencil $A_1-\lambda
A_2 - \mu B\succ0$ for some $\mu\in R$. This is also known to be the
dual Slater condition \cite{Y}, the stability condition \cite{MJ},
or the ``regular (ease)'' case \cite{SW,Pong}. We notice that, while
the primal Slater condition is quite natural and easy to satisfy,
the dual Slater condition is very strict. A sufficient condition for
the dual Slater condition is that at least one of the matrices
$A_1-\lambda A_2$ and $B$ is positive definite. A necessary
condition is that $A_1-\lambda A_2$ and $B$ can be simultaneously
diagonalizable via congruence (SDC). Namely, there exists a
nonsingular matrix $C$ (depending on $\lambda$) such that both
matrices $C^T(A_1-\lambda A_2)C$ and $C^TBC$ are diagonal.
Therefore, assuming the dual Slater condition for each $(P)_\lambda$
is impractical. Nevertheless, there were some papers which solve
quadratically constrained quadratic fractional problem using the
iterative method. For example, Beck et al. \cite{ABT} considered a
special case of (P) with $g(x)=x^T Bx,~B\succ 0$, $u\ge0,~v>0$.
Zhang and Hayashi \cite{ALi} studied a CDT-type quadratic fractional
problem subject to two quadratic constraints, one of which is a
ball, by an iterative generalized Newton method for finding
$f(\lambda_0)=0.$

On the other hand, $\lambda^*$ could be directly computed via an
exact semi-definite reformulation (SDR), rather than iteratively. In
particular, Beck and Teboulle \cite{ASY} considered an one-sided
homogeneous constrained quadratic problem below:
\begin{equation}\label{RQ}
{\rm (RQ)}~\inf \left\{ \frac{f_1(x)}{f_2(x)} : ||Lx||^2\le \rho \right\}
\end{equation}
where $L\in \Bbb R^{r\times n}$ is a full row rank matrix and
$\rho>0.$ Under some technical conditions, Problem (RQ) was shown to
possess a ``hidden convexity'' that it admits an exact SDR.
Therefore, the optimal value $\lambda^*$ can be evaluated in a
polynomial time. The result was later strengthened in \cite{AOS} by
Xia that the (RQ) problem indeed admits an exact SDR without any
condition. Moreover, it is attained if and only if the associated
SDR (\ref{SDP}) has a unique solution. Unfortunately, problems
beyond (RQ) are more complicate. An exact SDR is in general not
available for (P) even when $||Lx||^2\le \rho$ is relaxed to become
a convex nonhomogeneous constraint $g(x)=x^TBx+2d^Tx+\alpha.$ See
Example \ref{BB} in Sect. \ref{sec:3} for an explanation.

Later, Beck and Teboulle proposed a framework that minimizes the
ratio of two quadratic functions over $m$ quadratic inequalities
\cite{BT}:
\begin{equation}\label{QCRQ}
{\rm (QCRQ)}~\inf_{x\in \Bbb R^n} \left\{
\frac{f_1(x)}{f_2(x)} : g_i(x)=x^TB_ix+2d_i^Tx+\alpha_i\le 0,
i=1,2,..,m  \right\}.
\end{equation}
%
It covers quadratically constrained quadratic programming (QPQC) as
a special case. It is known that (QPQC) is NP-hard and there is no
surprise that an even more generic (QCRQ) can be studied only under
very restrictive situations. Based on the homogenization technique,
(QCRQ) can be made homogeneous by substituting
$x=\frac{y}{t},~t\not=0:$
\begin{align}\label{1}
\inf \left\{
\frac{f_1^H(y,t)}{f_2^H(y,t)} : g_i^H(y,t)\le 0, i=1,2,..,m  , t\ne0\right\},
\end{align}
where $f_1^H(y,t), f_2^H(y,t), g_i^H(y,t),~i=1,2,...,m$ are
homogeneous versions of $f_1(x),$ $f_2(x), g_i(x), i=1,2,...,m,$
respectively. Notice that the homogenization yields Problem
(\ref{1}) which is valid only for $t\not=0$, but the non-triviality
occurs normally in the case $t=0$ when homogenizing a quadratic
system. Beck and Teboulle further relaxed $t\ne0$ to be
$(y,t)\ne(0,0)$ and considered a slightly different ``mutated''
problem
\begin{align}\label{2}
\inf_{(y,t)\ne(0,0)} \left\{
\frac{f_1^H(y,t)}{f_2^H(y,t)} : g_i^H(y,t)\le 0, i=1,2,..,m  \right\},
\end{align}
By imposing $f_2^H(y,t)=1,$ (\ref{2}) was proven to be equivalent to
the following non-fractional problem:
\begin{equation}\label{3}
(H)\ \ \
\begin{array}{lll}
v(H)=&\min& f_1^H(z,s)\\
&\text{s.t. } &  f_2^H(z,s)=1\\
& &g_i^H(z,s)\le 0
\end{array}
\end{equation}
where $v(\cdot)$ denotes the optimal value of the problem $(\cdot).$
Restricting $s=0$ in (\ref{3}), a related problem $(H_0)$ is used as
a reference to be compared with $(H)$:
\begin{equation}\label{4} (H_0)\ \ \ \
\begin{array}{lll}
v(H_0)=&\min& f_1^H(z,0)\\
&\text{s.t. } &  f_2^H(z,0)=1\\
& &g_i^H(z,0)\le 0.
\end{array}
\end{equation}
Then, (QCRQ) was shown to have a tight semi-definite relaxation
under the following three conditions:
\begin{equation}\label{bt-1}
\left(\begin{matrix}A_2&b_2\\b_2^T&c_2\end{matrix}\right)\succ0;
\end{equation}
\begin{equation}\label{bt-2}
v(H) < v(H_0);
\end{equation}
\begin{equation}\label{bt-3}
\hbox{The semi-definite relaxation admits a rank-one optimal solution.}
\end{equation}

As we shall see later, the three assumptions
(\ref{bt-1})-(\ref{bt-3}) put (QCRQ) in a very rigid class. In Sect.
\ref{sec:2.3}, we provide two examples of (P), Examples
\ref{example2} and \ref{VD}, which violate at least (\ref{bt-1}) and
(\ref{bt-2}) but can be solved by our method. The drawback of the
direct method for finding $\lambda^*$ ``once for all'' lies on the
fact that there are not too many special cases of (P) that possess a
hidden convexity. More sophisticated analysis is often necessary.

Our idea to compute $\lambda^*$ relies on a new S-Lemma. See Sect.
\ref{sec:2.2} Theorem \ref{L3}. When the optimal solution $x^*$ is
an interior point of $X=\{u\le g(x)\le v\}$, the case is somehow
simple and we show that $\lambda^*$ can be computed by an SDP. See
Sect. \ref{sec:2.1} Theorem \ref{case1}. Otherwise, $x^*$ resides on
one of the two boundaries satisfying $g(x^*)=u$ or $g(x^*)=v$. In
either case, to find $\lambda^*$, one faces a parametric family of
one equality-constrained quadratic programming problems (\ref{b1}).
By a coordinate change, we need to only consider $h(x)=x^T Bx+2d^T
x=0$. Then, we can compute $\lambda^*$ also by an SDP provided the
family (\ref{b1}) can be converted to the other one (\ref{b2}):
\begin{align}
\lambda^*&=\inf_{x\in \Bbb R^n}\left\{\frac{f_1(x)}{f_2(x)}| h(x)=0\right\}\nonumber\\
&=\sup\left\{\lambda:\Big\{x\in \Bbb
R^n|\lambda>\frac{f_1(x)}{f_2(x)},  h(x)=0\Big\}=\emptyset\right\}\nonumber\\
&=\sup\left\{\lambda:\Big\{x\in \Bbb R^n|f_1(x)-\lambda
f_2(x)<0, h(x)=0\Big\}=\emptyset\right\}\label{b1}\\
&=\sup\left\{\lambda:f_1(x)-\lambda
f_2(x)+\mu h(x)\ge0,\forall x\in \Bbb R^n,\mu\in \Bbb R\right\}\label{b2}\\
&=\sup_{\lambda,\mu\in \Bbb R}
\left\{\lambda:\left(\begin{matrix}A_1-\lambda A_2+\mu B&b_1-\lambda b_2+\mu d\label{b3}\\
b_1^T-\lambda b_2^T+\mu d^T& c_1-\lambda c_2
\end{matrix}\right)\succeq0\right\}.
\end{align}
The equivalence of (\ref{b1}) and (\ref{b2}) is indeed a very strong
statement since it requires the S-lemma of the equality version to
hold {\it for every parameter} $\lambda$. Moreover, since $h(x)=0$
cannot have any Slater point, this variant of S-Lemma is more
difficult to obtain than the inequality version with $h(x)\le0$. We
show that, under the following constraint qualification for equality
constraint:
\\
\ \\
{\bf Assumption B}\hspace*{0.5cm}There exists $\zeta\in X=\{x\in
\Bbb R^n: h(x)=0\}$ such that
$$
 x^TBx=0 \Rightarrow (B\zeta+d)^Tx=0,
$$
(\ref{b1}) and (\ref{b2}) can be made equivalence. In addition,
Assumption B has an important feature that it relates merely to
$g(x)$ (or $h(x)$), not to the parametric family of functions
$f_1(x)-\lambda f_2(x)$. In contrast, the extended Finsler's theorem
[\cite{BE},Thm A.2] can not apply as it requires a condition (in our
format and notations)
\begin{align}\label{ss11}
A_1-\lambda A_2\succeq\eta B,~for~some\ \eta\in R
\end{align}
in which $\eta$ might vary for different $\lambda$'s. We provide an
example, Sect. \ref{sec:2.2} Remark (3.4), which can be solved by
our extended S-lemma, while there is no $\eta$ satisfying
(\ref{ss11}) right at the optimal value $\lambda^*$. Assumption B is
also more general than a condition imposed in \cite{D} Prop. 3.1,
where $h(x)$ was assumed to be strictly convex or strictly concave.
Compared with the dual Slater condition, Assumption B is easier to
obtain. For example, if $A_1-\lambda A_2=0$ and $B$ is positive
semidefinite, singular and $d$ is in the range of $B$, then
Assumption B can be satisfied while the dual Slater condition is
obviously violated. Consequently, some hard cases of (I-GTRS) that
can not be solved due to lack of a positive definite matrix pencil
can now be done under Assumption B.

The paper is organized as follows. In Section \ref{sec:2}, we study
Problem (P) under Assumption A. The first step of our
algorithm tries to determine whether the optimal solution $x^*$
could lie in the interior of $X$, followed by checks on both
boundaries otherwise. For each inspection, we use an SDP to compute a
potential $\lambda^*$ and then verify whether $f(\lambda^*)=0$ by
solving a (constrained) quadratic programming problem
$(P)_{\lambda^*}$. We show that (P) can be solved in polynomial time
under the constraint qualification Assumption B, which is
independent of the usual primal and dual Slater conditions. In
Section \ref{sec:3}, the one-sided (P) for which Assumption A is violated is treated. Our result is that the one-sided (P) can
be completely solved in polynomial time without any condition. An
interesting comparison between the two-sided original (P) and the
one-sided case  is elaborated in Remark
\ref{comparision}. The (RQ) problem (\ref{RQ}) as a special case of
the one-sided (P) can now be resolved without any technical
conditions. In Section \ref{sec:4}, we characterize conditions for
the ultimate assumption of (P) that the denominator function
$f_2(x)>0$ on $X$ such that (P) is well-defined. It turns out the
well-definedness property can be related to simultaneous
diagonalization via congruence. The final section concludes the
paper.

\section{Quadratic Fractional programming problem with two-sided quadratic inequality constraint}
\label{sec:2} In this section, we first characterize conditions
under which (P) is bounded from below and under which (P) can be
attained. Then, we show how to compute $\lambda^*$ using a
semi-definite programming approach. Some difficult cases of (P) are
resolved with the help of a new version of S-Lemma under Assumption
B, which is more powerful than the primal/dual Slater condition; a
similar result in \cite{D} Prop. 3.1; and the extended Finsler's
theorem [\cite{BE},Thm A.2]. Examples are given to illustrate all
the ideas.

\subsection{Boundedness, attainment, and unconstrained cases}\label{sec:2.1}
In fractional programming, it is often assumed that the feasible set
$X$ is compact. In general, a well-defined (P) is not necessarily
bounded from below and can not be always attained. The following two
lemmas, generalizing some basic results in fractional programming,
characterize completely the boundedness and the attainment
properties of (P) without the compactness assumption. We omit the
proof as the original compactness assumption was only used to
guarantee that the optimal value of (P) is attained and each
iteration of the Dinkelbach method is defined. The reader can refer
to Dinkelbach's original proof \cite{Dink} or a more general
discussion on a multi-ratios case. See for example
\cite{Crou,CroFer,Ber,Bar,Ju}.
\begin{lemma}[The boundedness problem]\label{B1} Suppose that (P) is well defined.
It is bounded below if and only if there exists a $\bar{\lambda}\in
\Bbb R$ such that $f(\bar{\lambda})\ge 0.$ Furthermore, if
$\lambda^*>-\infty,$ then
$$\lambda^*=\max_{f(\lambda)\ge0}\lambda.$$
\end{lemma}
%

The following Example \ref{EX3} shows that it is possible for a
bounded (P) to have $f(\lambda^*)>0$, in which case (P) is
unattainable. That is, the optimal value $\lambda^*$ can not be
attained.

\begin{exam}\label{EX3} It is easy to check
$$\lambda^*= \inf_{x\in \Bbb
R^3}\left\{\frac{x_1^2+1}{x_2^2+1}:g(x)=x_1^2+2x_3-1\le
0\right\}=0$$ by letting $x_1=0$ and $x_2$ go to infinity.
Solving its parametric problem
\begin{eqnarray*}
\begin{array}{ll}
f(\lambda)&=\inf_{x\in \Bbb R^3}
\Big\{x_1^2+1-\lambda(x_2^2+1):g(x)=x_1^2+2x_3-1\le 0 \Big\}\\
&=\left\{\begin{array}{ll}1-\lambda, &\text{ if  } \lambda\le0\\
-\infty, &\text{ if } \lambda >0\end{array},\right.
\end{array}
\end{eqnarray*}
we observe that $f(\lambda^*)=f(0)=1>0.$
\end{exam}

\begin{lemma}[The attainment problem]\label{B2}
Suppose that (P) is well defined. Then, $\lambda^*=v(P)$ is attained
at $x^*\in X$ if and only if $\lambda^*$ is a root of $f(\lambda)=0$
and $x^*$ is an optimal solution to $\rm{(P)}_{\lambda^*}.$
\end{lemma}
%

\begin{rem}\label{rem1}
In fact, Lemmas \ref{B1} and \ref{B2} hold for any well-defined
fractional programming problem where the ratio of functions are not
necessary to be quadratics, and the constraint set $X$ can be
arbitrary.
\end{rem}

\begin{rem}\label{rem1-1}
Due to Lemma \ref{B2}, we can freely exchange and mention the two
types of problems: either (P) or $(P)_{\lambda^*}$ with
$f(\lambda^*)=0$.
\end{rem}

In the following until the end of the section, we assume that
problem (P) is always attained and satisfies Assumption A. All other
cases not satisfying this assumption can be treated separately. When
$u\le\inf_{x\in \Bbb R^n}g(x),$ the constraint $u\le g(x)\le v$ is
reduced to $g(x)\le v.$ Similarly, if $\sup_{x\in \Bbb R^n}g(x)\le
v,$ then $u\le g(x)\le v$ becomes just $u\le g(x).$ The two cases
where $g(x)\le v$ or $u\le g(x)$ will be studied in next section.

If $u\le \inf g(x)\le\sup g(x)\le v,$ problem (P) is an
unconstrained quadratic fractional programming problem. According to
Lemma \ref{B1}, the optimal value can be computed directly by
\begin{eqnarray*}
\begin{array}{ll}
\lambda^*&=\max\limits_{f(\lambda)\ge0}\lambda\\
&=\max\{\lambda\in \Bbb R: \inf_{x\in \Bbb R^n} f_1(x)-\lambda f_2(x)\ge0\}\\
&=\max\{\lambda\in \Bbb R: f_1(x)-\lambda f_2(x)\ge0,~\forall x\in \Bbb R^n\}\\
&=\max\left\{\lambda\in \Bbb
R: \left(\begin{matrix} A_1-\lambda A_2&
b_1-\lambda b_2\\
b_1^T-\lambda b_2^T&c_1-\lambda
c_2\end{matrix}\right)\succeq0\right\}.
\end{array}
\end{eqnarray*}
It indicates that an unconstrained (P), if not unbounded below, must
be equivalent to the convex unconstrained problem: $\inf_{x\in \Bbb
R^n} f_1(x)-\lambda^*f_2(x)$. If (P) is attained, from Lemma
\ref{B2} the optimal solution $x^*$ can be also found by solving
$(P)_{\lambda^*}$.

Now we elaborate how to solve (P) under Assumption A. First notice
that the optimal solution $x^*$ of $\rm{(P)}_{\lambda^*}$ will be
either an interior point of $X=\{u\le g(x)\le v\}$, or resides on
one of the two boundaries satisfying $g(x^*)=u$ or $g(x^*)=v$. By
the first order and the second order necessary conditions, $x^*$ is
an interior point only when $A_1-\lambda^*A_2\succeq0$ and also
satisfies $(A_1-\lambda^*A_2)x^*=-(b_1-\lambda^*b_2)$. Therefore,
problem (P) can be analyzed by the following three (possibly overlapped) cases.\\
{\bf Case 1.} $f(\lambda^*)=0$, \emph{ $A_1-\lambda^*A_2\succeq0$
and
\begin{equation}\label{cm1}{x^*}\in S=\{x\in \Bbb
R^n|(A_1-\lambda^*A_2)x=-(b_1-\lambda^*b_2)\}\end{equation}
satisfying $u\le g(x^*)\le v.$}\\
{\bf Case 2.} $f(\lambda^*)=0$, and $x^*$ solves
\begin{equation}\label{cm2}
\begin{array}{ll}
\inf_{x\in \Bbb R^n} & f_1(x)-\lambda^* f_2(x)\\
\text{s.t. } &  g(x)= u
\end{array}
\end{equation}\\
{\bf Case 3.} $f(\lambda^*)=0$, and $x^*$ solves
\begin{equation}\label{cm3}
\begin{array}{ll}
\inf_{x\in \Bbb R^n} & f_1(x)-\lambda^* f_2(x)\\
\text{s.t. } &  g(x)=v
\end{array}
\end{equation}\\

Theorem \ref{case1} below shows that Case 1 can be directly solved
as if an unconstrained (P), while Case 2 and Case 3 must do, namely
to find $\lambda^*$ and $x^*$, with a new version of S-Lemma. See
Theorem \ref{L3} below.

\begin{theorem}\label{case1}
If $(\lambda^*,x^*)$ happens to satisfy Case 1, then $\lambda^*$ can
be computed by
\begin{align}\label{cv}
\lambda^*=\sup \left\{\lambda\in \Bbb
R:\left(\begin{matrix}A_1-\lambda
A_2&b_1-\lambda b_2\\
b_1^T-\lambda b_2^T& c_1-\lambda c_2
\end{matrix}\right)\succeq0\right\}.
\end{align}
\end{theorem}
\begin{proof}
Since Case 1 is assumed, $f_1(x)-\lambda^* f_2(x)$ is convex. By
(\ref{cm1}), $x^*$ is also a global minimizer of the unconstrained
quadratic problem
$$f(\lambda^*)=\inf_{x\in \Bbb R^n} f_1(x)-\lambda^*f_2(x)=0.$$
Then, we have $f_1(x)-\lambda^*f_2(x)\ge0, \forall x\in \Bbb R^n,$
which is equivalent to
$$\left(\begin{matrix}A_1-\lambda^*
A_2&b_1-\lambda^* b_2\\
b_1^T-\lambda^* b_2^T& c_1-\lambda^* c_2
\end{matrix}\right)\succeq0.$$
To show that $\lambda^*$ is the largest one satisfying the matrix
inequality (\ref{cv}), we suppose that there exists
$\bar{\lambda}>\lambda^*$ also satisfying that matrix inequality:
$$\left(\begin{matrix}A_1-\bar{\lambda}
A_2&b_1-\bar{\lambda} b_2\\
b_1^T-\bar{\lambda} b_2^T& c_1-\bar{\lambda} c_2
\end{matrix}\right)\succeq0.$$
Then
$$f_1(x)-\bar{\lambda}f_2(x)\ge0, \forall x\in \Bbb R^n.$$
Equivalently,
$$\frac{f_1(x)}{f_2(x)}\ge\bar{\lambda}>\lambda^*, \forall x\in \Bbb R^n,$$
which indicates that $f(\lambda^*)>0,$ a contradiction.
\end{proof}

To apply Theorem \ref{case1}, we first solve the SDP problem
(\ref{cv}) to get a candidate value $\lambda^*$. If
$f(\lambda^*)=0$, it has satisfied the first criterion in Case 1.
Since $A_1-\lambda^* A_2\succeq0$, $(P)_{\lambda^*}$ is convex. Any
unconstrained optimizer $x^*\in S$ satisfying $u\le g(x^*)\le v$
must also solve $(P)_{\lambda^*}$. Then, the value $\lambda^*$
computed by (\ref{cv}) is the optimal value of (P) with the optimal
solution $x^*$. Otherwise, if either $f(\lambda^*)>0$ or there is no
such $x^*\in S$ satisfying $u\le g(x^*)\le v$, Case 1 does not
happen and we have to look for Case 2 and Case 3. That is,
\begin{eqnarray}
\begin{array}{ll}
\lambda^*&=\inf\left\{\frac{f_1(x)}{f_2(x)}|~x\in X\right\}\\
&=\min\left\{\lambda_1\equiv\inf\{\frac{f_1(x)}{f_2(x)}|~g(x)=u\};\
\lambda_2\equiv\inf\{\frac{f_1(x)}{f_2(x)}|~g(x)=v\}\right\}.
\label{equivalue}
\end{array}
\end{eqnarray}

It is possible that at least one of $\lambda_1$ and $\lambda_2$ is
negative infinity, then (P) is unbounded below. Another possibility
is that, in (\ref{equivalue}), we have $\lambda_1=\lambda_2$. Then,
we need to check additionally which one,
$f(\lambda_1)=\inf\limits_{g(x)=u}\{f_1(x)-\lambda_1 f_2(x)\}=0$ or
$f(\lambda_2)=\inf\limits_{g(x)=v}\{f_1(x)-\lambda_2 f_2(x)\}=0$,
holds in order to determine $\lambda^*$. If neither $f(\lambda_1)=0$
nor $f(\lambda_2)=0$, (P) is unattainable according to Lemma
\ref{B2}. In the following subsection, we focus on solving the
quadratic fractional programming problem subject to one quadratic
equality constraint.

\subsection{An extended S-Lemma with equality}\label{sec:2.2}
Since Case 2 and Case 3 have the same pattern, we only discuss Case
2 in which $x^*$ satisfies $g(x^*)=u$. Without loss of generality,
we define $h(x)=g(x)-u$ and assume that $h(0)=0$. Otherwise, by
replacing $x$ with $x+x'$ for some nonzero vector $x'\in \{x\in \Bbb
R^n:g(x)=u\}$, the change of coordinate makes $\bar{h}(x)=h(x+x')$
satisfy $\bar{h}(0)=h(x')=0.$ Then, the problem casted in the new
coordinate system
\begin{equation}\label{cm2new}
\begin{array}{ll}
\inf_{x\in \Bbb R^n} & \left[\bar{f}_1(x)-\lambda^* \bar{f}_2(x)=
f_1(x+x')-\lambda^*
f_2(x+x')\right]\\
\text{s.t. } &  \bar{h}(x)=h(x+x')=0
\end{array}
\end{equation}
is equivalent to (\ref{cm2}) in the sense that if $x^*$ is an
optimal solution of (\ref{cm2new}), then $x^*+x'$ is optimal to
(\ref{cm2}). Conversely, if $x^*$ is optimal to (\ref{cm2}),
$x^*-x'$ is optimal to (\ref{cm2new}). Therefore, we only have to
deal with
\begin{equation}\label{cm22}
\begin{array}{lll}
f(\lambda^*)=&\inf_{x\in \Bbb R^n} & f_1(x)-\lambda^* f_2(x)\\
& \text{s.t. } &  h(x)=0
\end{array}
\end{equation}
where $h(x)=x^TBx+2d^Tx$ and $\lambda^*$ is the optimal value of the
following problem:
\begin{equation}\label{cm222}
\begin{array}{lll}
\lambda^*=&\inf_{x\in \Bbb R^n} & \frac{f_1(x)}{f_2(x)}\\
& \text{s.t. } &  h(x)=0.
\end{array}
\end{equation}
Theorem \ref{L3} below is an extended version of S-Lemma which,
under Assumption B, converts the fractional programming problem
(\ref{cm222}) to an equivalent SDP problem (\ref{uv}). Assumption B
plays the role of constraint qualification, which used to be the
primal Slater condition when an inequality system $f_1(x)-\lambda
f_2(x)<0, h(x)\le0$ is otherwise considered. Naturally, $h(x)=0$
does not possess any Slater point so that another type of constraint
qualification like Assumption B is needed.

\begin{theorem}[Extended S-Lemma with equality]\label{L3}
Under Assumptions A and B, the following two statements are
equivalent for each given $\lambda$.\\
(i) The system
\begin{align}\label{slm}
f_1(x)-\lambda f_2(x)<0, h(x)=0
\end{align}
is unsolvable.\\
(ii) There exists $\mu\in \Bbb R$ such that
$$f_1(x)-\lambda f_2(x)+\mu h(x)\ge0,\ \forall x\in \Bbb R^n.$$
\end{theorem}
\begin{proof}
Notice that statement (ii) trivially implies statement (i) without
any condition. We only prove for the converse under Assumptions A
and B. The proof will be presented in two cases: either $B\zeta=0$ or $B\zeta\not=0.$\\
${\bf(a)}$  $B\zeta=0.$ Then Assumption B becomes
\begin{align}\label{S-lem}
x^TBx=0 \Rightarrow d^Tx=0.
\end{align}

We first rewrite system (\ref{slm}) as
\begin{subequations}\label{sl1}
\begin{align}
&x^T(A_1-\lambda A_2)x+2(b^T_1-\lambda b_2^T)x+c_1-\lambda c_2  & <0;\label{s1}\\
&x^TBx+2d^Tx&\le 0;\label{s2}\\
&x^T(-B)x-2d^Tx&\le 0,\label{s3}
\end{align}
\end{subequations}
which can be made homogeneous by introducing a new variable $t\in
\Bbb R$ as follows:
\begin{subequations}\label{sl2}
\begin{align}
&x^T(A_1-\lambda A_2)x+2(b^T_1-\lambda b_2^T)xt+(c_1-\lambda c_2)t^2  & <0;\label{s4}\\
&x^TBx+2d^Txt&\le 0;\label{s5}\\
&x^T(-B)x-2d^Txt&\le 0.\label{s6}
\end{align}
\end{subequations}
We want to assert that if (\ref{sl1}) is unsolvable, then
(\ref{sl2}) is
unsolvable too. Suppose in contrary that (\ref{sl2}) has a solution $(\bar{x},\bar{t}).$\\
If $\bar{t}\ne 0,$ by dividing both sides of (\ref{s4})-(\ref{s6})
by $\bar{t}^2,$ we see that $\frac{\bar{x}}{\bar{t}}$ is a solution
to (\ref{sl1}), which is a
contradiction.\\
If $\bar{t}=0,$ system (\ref{sl2}) becomes
\begin{subequations}\label{sl3}
\begin{align}
&\bar{x}^T(A_1-\lambda A_2)\bar{x}  & <0;\label{s7}\\
&\bar{x}^TB\bar{x}&\le 0;\label{s8}\\
&\bar{x}^T(-B)\bar{x}&\le 0.\label{s9}
\end{align}
\end{subequations}
Inequalities  (\ref{s8}) and (\ref{s9}) together  imply that
$\bar{x}^TB\bar{x}=0.$ According to (\ref{S-lem}), $d^T\bar{x}=0$
and thus (\ref{s2}) and (\ref{s3}) are satisfied by $\bar{x}.$
Moreover, we observe that $\beta\bar{x}$ is a solution to
(\ref{sl3}) for any $\beta>0.$ By (\ref{s7}), $\bar{x}^T(A_1-\lambda
A_2)\bar{x} <0,$ we can then choose $\beta$ large enough such that
$\beta\bar{x}$ satisfies (\ref{s1}), and also (\ref{s2})-(\ref{s3}).
Therefore, if the system (\ref{sl1}) does not have a solution, the
homogeneous system (\ref{sl2}) must be also unsolvable.

The system (\ref{sl2}) can be put into the quadratic form as
follows:
$$
\begin{array}{ll}
y^TPy  & <0\\
y^TQy&\le 0\\
y^T(-Q)y&\le 0
\end{array}
$$
with $P=\left(\begin{matrix}A_1-\lambda A_2&b_1-\lambda
b_2\\b_1^T-\lambda b_2^T& c_1-\lambda c_2 \end{matrix}\right)$,
$Q=\left(\begin{matrix}B&d\\d^T&0 \end{matrix}\right)$ and
$y=\left(\begin{matrix}x\\t \end{matrix}\right)\in \Bbb R^{n+1}.$ It
is unsolvable if and only if
$$\underbrace{\left\{(y^TPy,y^TQy): y\in \Bbb R^{n+1}\right\}}_{C}\cap
\underbrace{\left\{(a,b)\in \Bbb R^{2}: a<0,
b=0\right\}}_{D}=\emptyset.$$ Notice that $C$ is convex (see, e.g.
\cite{Dines}),  $0\in C,$ and $D$ is also convex. If we express
$$D=\left\{z=\left(\begin{matrix}a\\b
\end{matrix}\right)\in \Bbb R^2: e_1^Tz<0, e_2^Tz\le0, e_3^Tz\le0\right\}$$
with $e_1=\left(\begin{matrix}1\\0 \end{matrix}\right),
e_2=\left(\begin{matrix}0\\1
\end{matrix}\right),e_3=\left(\begin{matrix}0\\-1
\end{matrix}\right),$ by the separation theorem \cite{Baz}, there exist $\eta_1,
\eta_2\in \Bbb R, \eta_1^2+\eta_2^2\ne0,$ such that
\begin{align}\label{sl5}
\eta_1(y^TPy)+\eta_2(y^TQy)\ge 0, \ \forall y\in \Bbb R^{n+1}
\end{align}
and
\begin{align}\label{sl6}
(\eta_1, \eta_2)^T z\le 0,\ \forall z\in D.
\end{align}
Applying Farkas' Lemma \cite{Baz} to (\ref{sl6}), there exists
$(\xi_1, \xi_2,\xi_3)\ge0$ such that $$(\eta_1,
\eta_2)^T=\xi_1e_1^T+\xi_2e_2^T+\xi_3e_3^T.$$ Namely,
\begin{align}\label{sb}
\eta_1=\xi_1\ge0, \eta_2=\xi_2-\xi_3.
\end{align}
Substituting (\ref{sb})  into (\ref{sl5}), we have
\begin{align}\label{sl7}
\xi_1(y^TPy)+(\xi_2-\xi_3)(y^TQy)\ge 0, \ \forall y\in \Bbb R^{n+1}
\end{align}
and $\xi_1\ge0.$\\
If $\xi_1=0,$  we must have $\xi_2-\xi_3\ne 0$ since $\eta_1,
\eta_2$ can not be both zero. Then (\ref{sl7}) implies that either
$y^TQy\ge 0$ or $y^TQy\le 0, \ \forall y\in \Bbb R^{n+1}.$ That is,
$Q$ can not be indefinite. By  Assumption A, we have
$$\inf_{x\in \Bbb R^n}h(x)<0<\sup_{x\in \Bbb R^n}h(x)$$
and there exist $x',x''\in \Bbb R^n$ such that $h(x')<0$ and
$h(x'')>0.$ Let $y'=(x',1)$ and $y''=(x'',1)$ then $(y')^TQ(y')<0$
and $ (y'')^TQ(y'')>0.$ This contradiction leads to $\xi_1>0.$
Dividing $\xi_1>0$ throughout (\ref{sl7}) and letting
$\mu=\frac{\xi_2-\xi_3}{\xi_1},$ we obtain
$$y^TPy+\mu(y^TQy)\ge 0, \
\forall y\in \Bbb R^{n+1}.$$ In particular, for
$y=\left(\begin{matrix}x\\1\end{matrix}\right),\  x\in \Bbb R^n,$
there is
$$f_1(x)-\lambda f_2(x)+\mu h(x)\ge 0,$$
which shows the validity of statement (ii).

\noindent${\bf (b)}$ $B\zeta\ne0.$ In this case, letting
$z=x-\zeta$, we have
\begin{equation}\label{coef}
f_1(x)-\lambda
f_2(x)=f_1(z+\zeta)-\lambda f_2(z+\zeta)=z^TAz+ 2b^Tz+ c,
\end{equation} where
\begin{align}\nonumber
A=A_1-\lambda A_2,
b=A\zeta+b_1-\lambda b_2 \text{ and }
c=\zeta^TA\zeta+2b^T\zeta+c_1-\lambda c_2.
\end{align} Also,
\begin{equation}\label{coef1}
h(x)=h(z+\zeta)=\zeta^TB\zeta+2d^T\zeta +z^TBz+
2(B\zeta+d)^Tz=z^TBz+2(d')^Tz,
\end{equation}
where $\zeta^TB\zeta+2d^T\zeta=0$ since $\zeta\in X $  and
$d'=B\zeta+d.$ Obviously, System (\ref{slm}) is unsolvable if and
only if
\begin{align}\label{case2}
z^TAz+2b^Tz+c<0, z^TBz+2(d')^Tz=0
\end{align}
does not have a solution. Moreover, Assumption B says that, if
$z^TBz=0$, there must be $(d')^Tz=0$ (where $d'=B\zeta+d$). Therefore,
under Assumption B, if System (\ref{case2}) is unsolvable, we can
apply the proof for case {\bf(a)} to get $\mu'$ such that
$$z^TAz+2b^Tz+ c+\mu'(z^TBz+2(d')^Tz)\ge0,~ \forall z\in \Bbb R^n,$$
which is equivalent to
$$f_1(x)-\lambda f_2(x)+\mu' h(x)\ge0, \forall x\in \Bbb R^n$$
by (\ref{coef}) and (\ref{coef1}). This completes the proof of
Theorem \ref{L3}.
\end{proof}

Applying Theorem \ref{L3}, we can compute the optimal value
$\lambda^*$ of (\ref{cm222}) by the SDP problem (\ref{uv}) below. The
proof of Theorem \ref{B5} was already sketched in Sect. \ref{sec:1}
(\ref{b1})-(\ref{b3}) and thus will not be repeated here.

\begin{theorem}\label{B5} Under Assumptions A and B,
the optimal value of problem (\ref{cm222})
$$\lambda^*=\inf_{x\in \Bbb R^n}\left\{\frac{f_1(x)}{f_2(x)}| h(x)=x^TBx+2d^Tx=0\right\}$$
can be computed by
\begin{align}\label{uv}
\lambda^*=\sup_{\lambda,\mu\in \Bbb R}
\left\{\lambda:\left(\begin{matrix}A_1-\lambda A_2+\mu B&b_1-\lambda b_2+\mu d\\
b_1^T-\lambda b_2^T+\mu d^T& c_1-\lambda c_2
\end{matrix}\right)\succeq0\right\}.
\end{align}
\end{theorem}

Now we can compute potential values for both $\lambda_1$ and
$\lambda_2$ in (\ref{equivalue}) by the SDP problem (\ref{uv}), but
yet to check $f(\lambda_1)=0$ or $f(\lambda_2)=0$ (and also to find
the optimal solution). It requires to solve the quadratic fractional
problem with an equality quadratic constraint of type (\ref{cm22}).
Mor\'{e} (\cite{MJ},Thm 3.2) has shown that, under a constraint
qualification (similar to our Assumption A) and assuming that
$A_1-\lambda^* A_2\not=0$, every optimal solution $x^*$ of
(\ref{cm22}) admits a Lagrange multiplier $\mu^*$ with no duality
gap. However, we do not know in advance whether $x^*$ exists.
Moreover, computing the saddle point $(x^*,\lambda^*)$
algorithmically requires the existence of a positive definite matrix
pencil $A_1-\lambda^* A_2+\mu B\succ0$ for some $\mu\in R$. Failing
to have a positive definite matrix pencil leads to a difficult
unstable (\ref{cm22}) that a small perturbation could make
(\ref{cm22}) become unbounded below. Our new version of S-lemma
hence provides an alternative way to deal with (\ref{cm22}). We show
that, under Assumption B, (\ref{cm22}) admits the strong duality so
that the rank one decomposition \cite{SZ,D} can be applied to get
$x^*$. We first notice that Assumption A implies $\{x\in \Bbb
R^n:g(x)=v\}\ne \emptyset,$ or equivalently, problem (\ref{cm22}) is
feasible.

 \begin{theorem}\label{cor}
Under Assumptions A and B, the strong duality holds for problem
(\ref{cm22}).
 \end{theorem}
\begin{proof}
Based on the extended S-Lemma Theorem \ref{L3}, we can evaluate
$f(\lambda^*)$ of (\ref{cm22}) by an SDP as follows:
\begin{align}
f(\lambda^*)&=\inf_{x\in \Bbb R^n}\{f_1(x)-\lambda^*f_2(x)| h(x)=0\}\nonumber\\
&=\sup\left\{\nu\in \Bbb R|\{x\in \Bbb R^n:f_1(x)-\lambda^*
f_2(x)<\nu, h(x)=0\}=
\emptyset\right\}\nonumber\\
&=\sup\left\{\nu\in \Bbb R|f_1(x)-\lambda^* f_2(x)-\nu+\mu h(x)\ge0,
\forall x\in \Bbb R^n\right\}
\nonumber\\
&=\sup\limits_{\nu,\mu\in \Bbb R}\left\{\nu\in \Bbb R|\left(\begin{matrix}A_1-\lambda^*A_2+\mu B&b_1-\lambda^*b_2+\mu d\\
b_1^T-\lambda^*b_2^T+\mu
d^T&c_1-\lambda^*c_2-\nu\end{matrix}\right)\succeq0\right\}\label{qe2}.
\end{align}
Notice that (\ref{qe2}) is the SDP reformulation for the Lagrange
dual problem of (\ref{cm22}), e.g., see \cite{Wolk}. It means that
the strong duality holds for Problem (\ref{cm22}).
\end{proof}

Due to the strong duality of Theorem \ref{cor}, the conic dual
problem
of (\ref{qe2})\\
$\text{(SDR)}\hspace*{2.5cm}
\begin{array}{lll}
  &\inf  & M(f_1-\lambda^* f_2)\bullet Z\\
\ &\text{s.t. } & M(h)\bullet Z= 0\\
\ & \ & Z\succeq 0, I_{nn}\bullet Z=1
\end{array}
$ \\
is indeed a tight SDP relaxation of Problem (\ref{cm22}), where
$$M(f_1-\lambda^* f_2)=\left(\begin{matrix}A_1-\lambda^*
A_2&b_1-\lambda^* b_2
\\b_1^T-\lambda^* b_2^T&c_1-\lambda^* c_2\end{matrix}\right),
M(h)=\left(\begin{matrix}B&d\\d^T&0\end{matrix}\right), I_{nn}=
\left(\begin{matrix}\bf 0&0\\0^T&1\end{matrix}\right),$$  and $Z\in
S^{n+1}_+$, the set of $(n+1)\times (n+1)$ positive semi-definite
symmetric matrices. If (P) is attained, then one of the values
$\lambda_1$ and $\lambda_2$, together with its (SDR) is attained.
Then, the optimal solution $x^*$ of (\ref{cm22}) can be obtained by
applying the matrix rank-one decomposition procedure \cite{SZ,D} to
an optimal solution $Z^*$ of (SDR). The attainment, however, can be
assured mostly under the dual Slater condition.

\begin{rem}\label{rem2-1}
We comment on the applicability of Assumption B.
\begin{enumerate}
\item[(3.1)] If $B\succ 0$ or $B\prec 0$ as assumed in \cite{D} Prop.
3.1, then $x^TBx=0$ if and only if $x=0.$ Hence Assumption B is
trivially true.
\item[(3.2)] If $B\succeq 0$ or $B\preceq 0$ but not
definite, then Assumption B is equivalent to the fact that $d$
is in the range space of $B.$ Indeed, since $B\succeq0$ (the
case $B\preceq0$ is similarly considered), $x^TBx=0$ if and only
if $Bx=0.$ Suppose that $d$ is in the range space of $B$ such
that $d=Bz$ for some $z.$ Then, for all $\zeta\in X$, there is
$(B\zeta+d)^Tx=(B\zeta+Bz)^Tx=\zeta^TBx+z^TBx=0$ if $x$
satisfies $Bx=0$. That is, Assumption B holds. Conversely,
%
if Assumption B holds, then $(B\zeta+d)^Tx=0$ for all $x$ such that $Bx=0.$ We have $(B\zeta+d)^Tx=(B\zeta)^Tx+d^Tx=\zeta^TBx+d^Tx=0,$ where
$\zeta^TBx=0$ since $Bx=0.$ It implies that $d^Tx=0$ for all $x$ such that $Bx=0.$ That is, $d$ is in the orthogonal complementary space of the null space of $B,$ then $d$ must be in the range space of $B.$

\item[(3.3)]  Assumption B also covers cases in which $B$ is indefinite.
For example, let
$B=\left(\begin{matrix}1&0\\0&-1\end{matrix}\right)$ and
$d^T=(1,1).$ Then Assumption B holds if and only if there exists
$\zeta\in\{x: x_1^2-x_2^2+2x_1+2x_2=0\}$ such that $B\zeta+d=0.$
It happens that $\zeta=(-1\ 1)^T$ is the only possibility.
\item[(3.4)] The extended Finsler's theorem [\cite{BE},Thm A.2] is also
a version of S-Lemma with equality, but it can not be applied to
compute $\lambda^*$ because the condition (\ref{ss11}) is very
difficult to satisfy for all $\lambda$. For example,
$$
\lambda^*=\inf_{x\in \Bbb R^3}\left\{\frac{x_1^2+x_3^2+2}{x_1^2+x_3^2+1}:
h(x)=x_1^2-x_2^2+2x_1+2x_2=0\right\}
$$
can be solved by our S-lemma Theorem \ref{L3} to get
$\lambda^*=1$. However, with $\lambda^*=1$,
$f_1(x)-\lambda^*f_2(x)=1$ and
$B=\left(\begin{matrix}1&0&0\\0&-1&0\\0&0&0\end{matrix}\right)$
is indefinite. There is no $\eta$ satisfying condition
(\ref{ss11}) of the extended Finsler's theorem.
\end{enumerate}
\end{rem}

\subsection{Examples}\label{sec:2.3}

In this subsession, two examples are used to demonstrate the entire
procedure of our ideas to solve (P). From the examples, we can also
observe that, although (QCRQ) studied by Beck and Teboulle \cite{BT}
is the most generic framework for quadratic fractional programming
problems, their approach fails to solve both examples since the
conditions (\ref{bt-1})-(\ref{bt-3}) are too restrictive to be
satisfied.

\begin{exam}\label{example2} Solve
\begin{align}\label{EX8}
\inf_{x\in \Bbb R^3}\left\{\frac{x_1^2+x_3^2+2x_3}{x_2^2+1}:0\le
x_3^2+2x_3\le 3\right\}.
\end{align}
\end{exam}
We first notice that $f_2(x)=x_2^2+1>0$ and thus (\ref{EX8}) is
well-defined. Moreover, Assumption A holds. To solve (\ref{EX8}), we check Case 1 first.\\
\emph{Step 1.} Solve the SDP problem (\ref{cv}) to get a candidate
$\lambda^*=-1.$ Then, for this $\lambda^*,$
$$f(\lambda^*)=\inf_{x\in \Bbb R^3}\{f_1(x)-\lambda^*f_2(x)\}=\inf_{x\in \Bbb R^3}\{x_1^2+x_3^2+2x_3+x_2^2+1\}
=0,$$ which is attained at $\hat{x}=(0,0,-1)^T$. Since $\hat{x}$
does not satisfy the constraint $0\le x_3^2+2x_3\le 3,$ Case 1 does
not hold. We go to the next step.\\
\emph{Step 2.} At this step we need to check Assumption B and it is
indeed satisfied.
Solve the SDP problem (\ref{uv}) for two cases: $h(x)=g(x)-u$ and $h(x)=g(x)-v.$\\
$\bullet$\ \ $h(x)=g(x)-u=x_3^2+2x_3.$ Since $h(0)=0$ we do not need
to make any change of coordinate. An immediately result from solving
(\ref{uv})
 gives $\lambda_1=0.$\\
$\bullet$\ \   $h(x)=g(x)-v=x_3^2+2x_3-3.$ Since $h(0)\ne0,$ we
select
$$x'=(0,0,1)^T\in\{x\in \Bbb R^3:g(x)-v=x_3^2+2x_3-3=0\}$$
and make a coordinate change by replacing $x$ with $x+x'$ so that
$$f_1(x+x')=x_1^2+(x_3+1)^2+2(x_3+1)= x_1^2+x_3^2+4x_3+3;\  f_2(x+x')=x_2^2+1;$$ and
$h(x+x')=(x_3+1)^2+2(x_3+1)-3=x_3^2+4x_3.$ Solving the SDP
(\ref{uv}) we also get $\lambda_2=0.$\\ Since
$\lambda_1=\lambda_2=0$, we have to compute $f(\lambda_1)$ and
$f(\lambda_2)$ to see which one is 0. It turns out that
\begin{align}\label{vd8}
\begin{array}{llll}
0=&f(\lambda_1)=&\inf_{x\in \Bbb R^3}\ & x_1^2+x_3^2+2x_3\\
& \ & \text{  s.t.  } &  x_3^2+2x_3=0,
\end{array}
\end{align}
whereas
\begin{align}
\begin{array}{lll}
3=f(\lambda_2)=&\inf_{x\in \Bbb R^3}\ & x_1^2+x_3^2+4x_3+3\\
& \ \text{  s.t.  } &  x_3^2+4x_3=0.
\end{array}
\end{align}
Therefore, $\lambda^*=\lambda_1=0$ and the optimal solution set for
(\ref{vd8}) is $$X^*=\{(0,a,0)^T, (0,b,-2)^T|~ a,b\in \Bbb
R\}\subset \Bbb R^3.$$ Since there is no change of coordinate in
computing $\lambda_1$, the set $X^*$ is also the optimal solution
set for (\ref{EX8}).

\begin{rem}
In Example \ref{example2}, since $A_2$ is a singular positive
semi-definite matrix, Condition (\ref{bt-1}) is violated. After
homogenization, the related problems $(H)$ (\ref{3}) and $(H_0)$
(\ref{4}) are formulated as
\begin{equation}\nonumber
\begin{array}{ll} (H)\ \ \ \ \
\min& z_1^2+z_3^2+2z_3s\\
\text{s.t. } & z_2^2+s^2=1\\
&0\le z_3^2+2z_3s\le 3s^2;
\end{array}
\end{equation}
and
\begin{equation}\nonumber
\begin{array}{ll} (H_0)\ \ \ \ \
\min& z_1^2+z_2^2\\
\text{s.t. } & z_2^2=1\\
&0\le z_3^2\le 0.
\end{array}
\end{equation}
It is easy to see $v(H)=v(H_0).$ Then condition (\ref{bt-2}) is also
violated. In other words, Beck and Teboulle's algorithm proposed in
\cite{BT} can not be used to solve Example \ref{example2}.
\end{rem}

\begin{exam}\label{VD}
Let $X=\{(x_1, x_2)^T| 0\le x_2^2+2x_1\le 2\}$ and solve
\begin{align}\label{EX7}
\inf_{x\in \Bbb R^2}\left\{\frac{2x_2}{x_2^2+1}: x\in X\right\}.
\end{align}
\end{exam}
Again, (\ref{EX7}) is well-defined and Assumption A is satisfied.
However,  since $d=(1,0)^T$ is not in the range space of
$B=\left[\begin{matrix}0&0\\0&1\end{matrix}\right]\succeq0,$
Assumption B is violated. Fortunately, we will see that (\ref{EX7})
meets Case 1, which does not need Assumption B. To justify, we solve
the SDP problem (\ref{cv}) to get $\lambda^*=-1$ and find that
$A_1-\lambda^*A_2=\left[\begin{matrix}0&0\\0&1\end{matrix}\right]\succeq0.$
Moreover, since
$$f(\lambda^*)=f(-1)=\inf_{x\in \Bbb R^2} \{f_1(x)-\lambda^*f_2(x)=x_2^2+2x_2+1\}=0.$$
The stationary points of $f_1(x)-\lambda^*f_2(x)=x_2^2+2x_2+1$ is
defined as follows.
$$S=\{x=(x_1, x_2)^T\in \Bbb R^2: 2x_2+2=0\}=\{(x_1,-1)^T\}\subset \Bbb R^2.$$
Now the intersection  $S\cap X=\{(x_1,-1)^T: -\frac{1}{2}\le x_1\le
\frac{1}{2}\}$ is the optimal solution set of (\ref{EX7}) and
$\lambda^*=-1$ is the optimal value.

\begin{rem}
Since Example \ref{example2} and Example \ref{VD} share the same
$f_2(x)$, Condition (\ref{bt-1}) is again violated. Similarly, for
Example \ref{VD}, we have:
\begin{equation}\nonumber
\begin{array}{ll} (H)\ \ \ \ \
\min& 2z_2s\\
\text{s.t. } & z_2^2+s^2=1\\
&0\le z_2^2+2z_1s\le 2s^2
\end{array}
\end{equation}
and $(H_0):$
\begin{equation}\nonumber
\begin{array}{ll} (H_0)\ \ \ \ \
\min& 0\\
\text{s.t. } & z_2^2=1\\
&0\le z_2^2\le 0.
\end{array}
\end{equation}
Since Problem $(H_0)$ is infeasible, Condition (\ref{bt-2}) can not
be true. Beck and Teboulle's algorithm fails to solve Example
\ref{VD}.
\end{rem}

\section{ Quadratic Fractional Programming Problem with One Inequality Quadratic
Constraint (QF1QC)}
\label{sec:3}
As analyzed in Sect. \ref{sec:2.1}, when Assumption A is violated, Problem (P) becomes either an  unconstrained problem or having one-sided  constraint $g(x)\le v$ or $u\le g(x).$ The unconstrained quadratic fractional
programming problem is, in fact, equivalent to the convex unconstrained quadratic problem as studied in Sect. \ref{sec:2.1}. In this section we study Problem (P) with an one-sided quadratic constraint taking the  following form:
\begin{equation}\label{QF1QC}
{\rm (QF1QC)}~ \inf_{x\in \Bbb R^n}\left\{
\frac{f_1(x)}{f_2(x)} : g(x)\le0 \right\}
\end{equation}
where $f_1(x), f_2(x)$  and $g(x)$ are quadratic functions as
defined at the beginning of the paper. The parametric problem
$\rm{(P)}_\lambda$ is now reduced to a quadratic programming problem
having one quadratic inequality constraint (QP1QC):
\begin{equation}\label{QF1QC_lambda}
\begin{array}{lll}
{\rm{(QF1QC)}_\lambda} \hspace*{1.0cm} f(\lambda)=  &\inf_{x\in \Bbb R^n} & f_1(x)-\lambda f_2(x)\\
\ &\text{s.t. } & g(x)\le 0.
\end{array}
\end{equation}
Assume in this section that  problem (QF1QC) satisfies the Slater
condition, i.e., there exists $\bar{x}\in \Bbb R^n$ such that
$g(\bar{x})<0.$ Otherwise, the problem (QF1QC) is either infeasible
or reduced to an unconstrained fractional programming problem, the
latter of which has been discussed in Section 2.
\begin{lemma}\label{No-Slater}
If Problem (QF1QC) has no Slater point, it is either infeasible or
equivalent to an unconstrained quadratic fractional programming
problem.
\end{lemma}
\begin{proof}
The Slater condition is violated only when $g(x)\ge0, \forall x\in
\Bbb R^n.$ This implies that $B\succeq0,$ i.e,  $g(x)$ is convex,
and $d\in \mathcal{R}(B),$ where $\mathcal{R}(B)$ is the range space
of $B.$ That is, the affine space
$$\{x\in \Bbb R^n: Bx+d=0\}\ne\emptyset.$$
Then $Bx+d=0 \Leftrightarrow x=-B^+d+Wz,$ where $B^+$ is the
Moore-Penrose generalized inverse of $B$ and $W$ is a matrix whose
columns form a basis for the null space of $B$ if $B$ is singular;
and $W=0$ if $B$ is nonsingular. Since $g(x)$ is convex,
$x=-B^+d+Wz$ is the global minimizer of $g(x)$ with the minimum
value $-d^TB^+d+\alpha.$ If $-d^TB^+d+\alpha>0, g(x)\ge -d^TB^+d
+\alpha>0$ implies that (QF1QC) is infeasible. If $-d^TB^+d
+\alpha=0,$
 then $g(x)\ge0.$ In this case, the feasible domain $X=\{x| g(x)\le0\}$ is reduced to
$X=\{x| g(x)=0\}.$ That is
$$\left\{x\in \Bbb R^n: g(x)\le 0\right\}=
\left\{\begin{array}{ll}\{-B^+d+Wz,\  z\in \Bbb R^m\}, &\text{ if }
d^TB^+d=\alpha\\
\emptyset, & \text{ if } d^TB^+d<\alpha
\end{array}\right. $$
where $m$ is the dimension of the null space of $B.$ In the case
that (QF1QC) is feasible, it can be expressed in term of $z\in \Bbb
R^m$ and becomes the following unconstrained fractional programming
problem:
\begin{align}\label{ucstr}
\lambda^*=\inf_{x\in \Bbb R^n}\left\{\frac{f_1(x)}{f_2(x)}:g(x)\le
0\right\}=\inf_{z\in\Bbb R^m} \frac{\bar{f}_1(z)}{\bar{f}_2(z)},
\end{align}
where $\bar{f}_i(z)=f_i(-B^+d+Wz)=z^TQ_iz-2q_i^Tz+\gamma_i,\
Q_i=W^TA_iW,\ q_i^T=(d^TB^+A_i-b_i^T)W,\
\gamma_i=d^TB^+A_iB^+d-2b_i^TB^+d+c_i, i=1,2.$
\end{proof}
\begin{theorem}\label{B4}
For any well-defined problem (QF1QC) satisfying the Slater
condition, its optimal value $\lambda^*$ can be determined by
solving the following semi-definite programming problem
\begin{equation}\label{F1}
\lambda^*= \sup_{\lambda\in \Bbb R, \mu\ge0}
\left\{\lambda:\left(\begin{matrix}A_1-\lambda A_2+\mu B&b_1-\lambda b_2+\mu d\\
b_1^T-\lambda b_2^T+\mu d^T& c_1-\lambda
c_2+\mu\alpha\end{matrix}\right)\succeq0\right\}.
\end{equation}
\end{theorem}
\begin{proof}
We have
\begin{subequations}
\begin{align}
\lambda^*&=\inf_{x\in \Bbb R^n}\left\{\frac{f_1(x)}{f_2(x)}:g(x)\le 0\right\}\nonumber\\
&=\sup\left\{\lambda:\{x\in \Bbb
R^n|\lambda>\frac{f_1(x)}{f_2(x)}, g(x)\le 0\}=\emptyset\right\}\nonumber\\
&=\sup\left\{\lambda:\{x\in \Bbb R^n|f_1(x)-\lambda
f_2(x)<0,g(x)\le0\}=\emptyset\right\}\label{b}\\
&=\sup\left\{\lambda:f_1(x)-\lambda
f_2(x)+\mu g(x)\ge0,\forall x\in \Bbb R^n, \mu\ge0\right\}\label{c}\\
&=\sup\left\{\lambda:\left(\begin{matrix}A_1&b_1\\b_1^T&c_1\end{matrix}\right)
-\lambda\left(\begin{matrix}A_2&b_2\\b_2^T&c_2\end{matrix}\right)
+\mu\left(\begin{matrix}B&d\\d^T&\alpha\end{matrix}\right)\succeq
0, \mu\ge 0\right\}\nonumber\\
&=\sup_{\lambda\in \Bbb R, \mu\ge0}
\left\{\lambda:\left(\begin{matrix}A_1-\lambda A_2+\mu B&b_1-\lambda b_2+\mu d\\
b_1^T-\lambda b_2^T+\mu d^T& c_1-\lambda
c_2+\mu\alpha\end{matrix}\right)\succeq0\right\}\nonumber,
\end{align}
\end{subequations}
where the equivalence of (\ref{b}) and (\ref{c}) is due to a
standard S-lemma under the Slater condition \cite{D,Yaku}.
\end{proof}\\
To know whether (QF1QC) is attained and to find $x^*$ that solves
(QF1QC), we need to check whether $\lambda^*$ found in (\ref{F1})
satisfies $f(\lambda^*)=0$ and to solve $\rm{(QF1QC)}_{\lambda^*}.$  We
have
\begin{equation} f(\lambda^*)=\sup\left\{\nu\in \Bbb R:\{x\in
\Bbb R^n|f_1(x)-\lambda^* f_2(x)<\nu, g(x)\le
0\}=\emptyset\right\}.\label{po}
\end{equation}
Since the Slater condition is assumed, we can apply S-lemma to
(\ref{po}) and obtain
$$f(\lambda^*)=\sup\left\{\nu\in \Bbb R: f_1(x)-\lambda^* f_2(x)-\nu+\eta g(x)\ge 0, \forall x\in \Bbb R^n,
\eta\ge 0\right\},$$ which is equivalent to a convex SDP
formulation:
\begin{align}\label{SD}
 f(\lambda^*)=\sup\left\{\nu\in \Bbb R:
\left(\begin{matrix}A_1-\lambda^* A_2+\eta B&b_1-\lambda^* b_2+\eta
d\\b_1^T-\lambda^* b_2^T+\eta d^T &c_1-\lambda^*
c_2+\eta\alpha-\nu\end{matrix}\right)\succeq 0, \eta\ge 0\right\}.\
\end{align}
We notice that (\ref{SD}) is the Lagrange dual problem of
$\rm{(QF1QC)}_{\lambda^*}$ \cite{Wolk}. It means that, the strong duality
holds for $\rm{(QF1QC)}_{\lambda^*}.$ Therefore, $\rm{(QF1QC)}_{\lambda^*}$
has the following tight SDP relaxation:
\begin{equation}
\begin{array}{ll}
 \inf  & M(f_1-\lambda^* f_2)\bullet Z\\
\text{s.t. } & M(g)\bullet Z\le 0 \label{SDP-tight2}\\
 \ & Z\succeq 0, I_{nn}\bullet Z=1,
\end{array}
\end{equation}
where  $M(f_1-\lambda^* f_2), M(g),I_{nn} $ and $Z$ are similarly
defined as in Sect. \ref{sec:2}. Then an optimal solution $x^*$ of
$\rm{(QF1QC)}_{\lambda^*}$, if exists, can be obtained from an optimal
solution of (\ref{SDP-tight2}) followed by the matrix rank-one
decomposition procedure. See \cite{SZ,D}.

\begin{rem}\label{comparision}
In Sect. \ref{sec:2}, our analysis showed that the difficulty of the
two-sided (P) lies mainly on the equality constrained problem
(\ref{cm222}), which can only be solved under the constraint
qualification Assumption B. Interestingly, we also showed that the
one-sided (P), namely (QF1QC), can be solved completely without any
condition. This leads to a conclusion that the equality constrained
version is more difficult than its counterpart with an inequality
constraint. They are not identical, even though for each $\lambda$
the two-sided $(P)_{\lambda}$ (\ref{P_lambda}); the equality version
of $(P)_{\lambda}$ (\ref{cm2}); and the inequality version of
$(P)_{\lambda}$ (\ref{QF1QC_lambda}) all possess a set of (similar
in format, but the difficulty in solving them might differ)
necessary and sufficient conditions that guarantee a strong duality,
respectively in (\cite{Pong}, Thm 2.3); (\cite{MJ}, Thm 3.2); and
(\cite{MJ}, Thm 3.3). We have some reasons for it. Geometrically,
even for a convex $g(x)$, $g(x)\le0$ leads to a convex set whereas
$g(x)=0$ not. Technically, the S-lemma is crucial in both cases. For
the inequality version $g(x)\le0$, the proof of the S-lemma must
rely on the Slater point. Fortunately, when $g(x)\le0$ fails the
Slater condition, it leads to a fact that $g(x)$ must be convex as
shown in Lemma \ref{No-Slater}. On the other hand, the equality
version $g(x)=0$ can not have a Slater point. It must rely on a more
sophisticate constraint qualification like Assumption B. Failing
that constraint qualification does not conclude any convexity of
$g(x)$.
\end{rem}

In the remaining part of this section, we shall discuss the (RQ)
problem (\ref{RQ}) as a special case of (QF1QC), where $g(x)$ in
(RQ) is convex ($B\succeq0$); no linear term ($d=0$) and
$\alpha=-\rho<0$. Suppose the rank of matrix $L$ is $r$ such that
$0<r\le n.$ Due to the special property of (RQ), Lemma \ref{B2} and
Theorem \ref{B4} can be combined to have a stronger version as
Theorem \ref{Xia} below. To prove it, we quote and use a result from
\cite{A}, which states: Consider a quadratic problem
\begin{align}\label{S} \min_{x\in \Bbb R^n}
\Big\{ g_0(x): g_i(x)\le 0, i=1,2,...,m\Big\}
\end{align} where
$g_i(x)=x^TQ_ix+2q_i^Tx+c_i, i=0,1,2,...,m$ are quadratic functions.
Then,
\begin{lemma}(\cite{A},Theorem 2)\label{L}  Suppose that $Q_1\succeq0$ and $Q_i=0$ for $
i=2,...,m.$ Then  if the objective function $g_0(x)$ is bounded from
below over the feasible set $X=\{x\in \Bbb R^n: g_i(x)\le 0,
i=1,...,m\}$ then problem (\ref{S}) attains its minimum.
\end{lemma}
\begin{theorem} (The attainment of the (RQ) problem)\label{Xia}
For any well-defined (RQ), the following three statements are equivalent:\\
(i) $\lambda^*=v\rm{(RQ)}$ is attained.\\
(ii) The following semi-definite programming problem (D) has a
unique solution $(\lambda^*,\eta^*):$
\begin{align}\label{SDP}
(D)\ \ \ \max_{\lambda\in \Bbb R,\eta\ge 0}
 \ \left\{\lambda:\left(\begin{matrix}A_1&b_1\\b_1^T&c_1\end{matrix}\right)
-\lambda\left(\begin{matrix}A_2&b_2\\b_2^T&c_2\end{matrix}\right)
+\eta\left(\begin{matrix}L^TL&0\\0&-\rho\end{matrix}\right)\succeq
0\right\}.
\end{align}
(iii) $f(\lambda^*)=0.$
\end{theorem}
\begin{proof}
The equivalence of (i) and (ii) was indeed proved in \cite{AOS}
Theorem 3.3. It suffices to show the equivalence between (i) and
(iii), which strengthens the general result Lemma \ref{B2} that
$\lambda^*=v\rm{(RQ)}$ is attained if and only if $f(\lambda^*)=0$
and
$$\rm{(RQ)}_{\lambda^*}:f(\lambda^*)= \inf_{x\in \Bbb
R^n}\left\{f_1(x)-\lambda^* f_2(x):||Lx||^2\le \rho \right\}
$$
is attained. However, if $f(\lambda^*)=0,$ $\rm{(RQ)}_{\lambda^*}$
is bounded below. Since $||Lx||^2\le \rho$ is convex, Lemma \ref{L}
assures that a bounded $\rm{(RQ)}_{\lambda^*}$ must be attained. In
other words, $f(\lambda^*)=0$ implies the attainment of problem
$(RQ)_{\lambda^*}$.
\end{proof}

\begin{rem}
The equivalence of (i) and (ii) in Theorem \ref{Xia} only holds for
(RQ) problem. It can not be extended to (QF1QC) in general as the
following Example \ref{BB} shows.
\end{rem}

\begin{exam}\label{BB}
Consider Example \ref{EX3} again as follows: $$\inf_{x\in \Bbb R^3}
\left\{\frac{x_1^2+1}{x_2^2+1}:g(x)=x_1^2+2x_3-1\le 0\right\}.$$
\end{exam}
It has been verified in Example \ref{EX3} that $\lambda^*=0$ and
$f(\lambda^*)=f(0)=1>0$, so the problem is unattainable. However,
the SDP problem (\ref{SDP}):
   $$\max\left\{\lambda:
\left(\begin{matrix}1+\eta&0&0&0\\0&-\lambda&0&0\\0&0&0&\eta\\
0&0&\eta&1-\lambda-\eta\end{matrix}\right)\succeq 0, \
\eta\ge0\right\}$$ has a unique solution $(\lambda^*,\eta^*)=(0,0).$

Some similar results of the attainment of the (RQ) problem were also
discussed in \cite{ASY} under stricter conditions. For comparison,
we quote the conditions and the results from \cite{ASY}.\\ {\bf
Assumption C (\cite{ASY})} There exists $\eta\ge 0$ such that
$$\left(\begin{matrix}A_2&b_2\\b_2^T&c_2\end{matrix}\right)
+\eta\left(\begin{matrix}L^TL&0\\0&-\rho\end{matrix}\right)\succ 0$$\\
{\bf Assumption D (\cite{ASY})} Either ($r=n$) or ($r<n$ and
$\lambda_{\text{min}}(M_1,M_2)<\lambda_{\text{min}}(F^TA_1F,F^TA_2F)$)
where
$$M_1=\left(\begin{matrix}F^TA_1F&F^Tb_1\\b_1^TF&c_1\end{matrix}\right),\
M_2=\left(\begin{matrix}F^TA_2F&F^Tb_2\\b_2^TF&c_2\end{matrix}\right)$$
with $F\in \Bbb R^{n\times (n-r)}$ a matrix whose columns form an
orthonormal basis for the null space of $L.$\\
\begin{theorem}(\cite{ASY})\label{Beck1}
 If Assumptions C and D are satisfied, the minimum of (RQ) is attained and $v\rm{(RQ)}\le
\lambda_{\min}(M_1, M_2).$
\end{theorem}
\begin{theorem}(\cite{ASY})\label{Beck2}
Let $n\ge 2$ and suppose that  Assumptions C and D  are satisfied.
Then
$$v(\rm{D})=\lambda^*,$$
where (D) is the semi-definite problem (\ref{SDP}).
\end{theorem}

It was proved by Example 3.5 in \cite{AOS} that Assumption D is not
a necessary condition for the attainment of (RQ). Therefore, the
necessary and sufficient statements (i) and (ii) in Theorem
\ref{Xia} strictly generalize Theorem \ref{Beck1}. Since Example
\ref{BB} further shows that the equivalence of (i) and (ii) in
Theorem \ref{Xia} does not hold for (QF1QC), our Lemma \ref{B2} thus
improves Theorem \ref{Beck1} sharply. Secondly, our Theorem \ref{B4}
shows that the conclusion of Theorem \ref{Beck2} is indeed true for
a more general (QF1QC) problem without any condition.

\section{On  well-definedness of (QF1QC)}
\label{sec:4} In this section we characterize the well-definedness
property for the problem (QF1QC) (i.e. $f_2(x)>0$ on $X=\{x\in \Bbb
R^n| g(x)\le0\}$). To this end, we assume the primal Slater
condition and that the two matrices $A_2$ and $B$ are simultaneously
diagonalizable via congruence (SDC). Then, there exists a
nonsingular matrix $C$ such that both matrices $C^TA_2C$ and $C^TBC$
are diagonal. It has been argued in \cite{Sheu} that, under the
(SDC) condition, the following quadratic problem
\begin{equation} \label{M}
\inf_{x\in \Bbb R^n}\left\{f_2(x): g(x)\le 0\right\}
\end{equation}
would be either unbounded below, or attained, or eventually reduced
to an unconstrained quadratic problem, but can never be
unattainable.

However, if (\ref{M}) is indeed an unconstrained problem, it is
attainable if and only if $A_2\succeq0$ and the vector $b_2$ lies in
the range space of $A_2$; or it must be unbounded below. In other
words, under the SDC condition, (\ref{M}) can never be unattainable
if bounded from below. It leads to a similar, but slightly more
general result than Lemma \ref{L}:
\begin{lemma}\label{SDC}
If $A_2$ and $B$ are SDC, the quadratic problem (\ref{M})
is either attained or unbounded below.
\end{lemma}

\begin{rem}
The question as to ``Simultaneous diagonalization via congruence of
a finite collection of symmetric matrices'' was proposed to be the
twelfth open problem in \cite{Bap}. Even for just two matrices, the
complexity to check whether or not they are indeed SDC remains
unanswered.
\end{rem}


The well-definedness of (QF1QC) can be checked computationally as
follows.
\begin{theorem}\label{B3}
Suppose that $A_2$ and $B$ are SDC. The following three statements are equivalent under
the primal Slater condition:\\
(a) Problem (QF1QC) is well-defined. That is, $f_2(x)>0$ on $X=\{x\in \Bbb R^n| g(x)\le0\}.$\\
(b) There exist $\delta>0, \eta\ge 0$ such that
\begin{equation}\label{D1}
\left(\begin{matrix}A_2&b_2\\
b_2^T&c_2-\delta\end{matrix}\right)+
\eta\left(\begin{matrix}B&d\\
d^T&\alpha\end{matrix}\right)\succeq0.
\end{equation}
(c) There exists $\delta>0$ such that \[\inf_{x\in \Bbb R^n}
\{f_2(x):g(x)\le 0\}\ge \delta>0.
\]
\end{theorem}
\begin{proof}
We observe that $(c)$ trivially implies$ (a).$ It remains to show that
$(a)$ implies $(b)$ and $(b)$
implies $(c).$\\
$(a) \Rightarrow (b).$ Since (QF1QC) is well-defined, $f_2(x)$ is
bounded from below by 0 over $\{x\in \Bbb R^n: g(x)\le 0\}$ and thus
Problem (\ref{M}) attains its minimum, say at $x^*$, by Lemma
\ref{SDC}. Let $\delta=f_2(x^*)>0.$ The following system
$$\begin{matrix}f_2(x)<\delta\\ g(x)\le 0\end{matrix} $$
is hence unsolvable. By S-Lemma, there exists $\eta\ge0$ such that
$$f_2(x)-\delta+\eta g(x)\ge 0, \forall x\in \Bbb R^n,$$
which is exactly (\ref{D1}).\\
$(b)\Rightarrow (c).$ Suppose that there exist $\delta>0, \eta\ge0$
such that the matrix inequality (\ref{D1}) holds. Since $g(x)\le 0$
and $\eta\ge0,$ we have
$$\inf_{x\in \Bbb R^n}\{f_2(x): g(x)\le 0\}\ge
\inf_{x\in \Bbb R^n}\{f_2(x)+\eta g(x): g(x)\le 0 \}.$$ Moreover,
the matrix inequality (\ref{D1}) is equivalent to
$$f_2(x)-\delta+\eta g(x)\ge 0, \forall x\in \Bbb
R^n$$ so that
$$\inf_{x\in \Bbb R^n}\{f_2(x)+\eta g(x):~g(x)\le 0\}\ge \delta>0.$$ Thus $$\inf_{x\in \Bbb R^n}\{f_2(x):
g(x)\le0\}\ge \delta>0$$ which completes the proof.
\end{proof}
\begin{rem}
The assumption that $A_2$ and $B$ are SDC in Theorem \ref{B3} can
not be relaxed as the following example explains.
\end{rem}
\begin{exam}
 $\inf_{x\in \Bbb R^2}\left\{\frac{x_1^2+x_2^2+5}{x_1^2}:1-2x_1x_2\le 0\right\}.$
\end{exam}
We note here that the matrices $A_2=\left(\begin{matrix}1&0\\
0&0\end{matrix}\right)$ and $B=\left(\begin{matrix}0&-1\\
-1&0\end{matrix}\right)$ are not SDC. However, the problem is
well-defined since vector $(0,x_2)^T$ is not in the feasible set. We
can also verify that the matrix
$$\left(\begin{matrix}A_2&b_2\\
b_2^T&c_2-\delta\end{matrix}\right)+
\eta\left(\begin{matrix}B&d\\
d^T&\alpha\end{matrix}\right)=\left(\begin{matrix}1&-\eta&0\\
-\eta&0&0\\ 0&0&\eta-\delta\end{matrix}\right)$$ is not positive
semi-definite for any $ \eta\ge 0, \delta>0.$ The statement (b) in
Theorem \ref{B3} fails.

Finally, it was proved in \cite{ASY} that Assumption C implies the
well-definedness of Problem (RQ). The following example indicates
that (\ref{D1}) in Theorem \ref{B3} is more general than Assumption
C.

\begin{exam}\label{EX9}
\begin{align}\label{wdf}
\ \inf_{x\in \Bbb R^3}
\Big\{\frac{x_1^2+x_2^2+x_3}{x_1^2+1}:x_1^2+x_2^2\le 1\Big\}
\end{align}
\end{exam}
In can be observed that (\ref{wdf}) is an (RQ) problem which
satisfies the SDC condition. It is well-defined so that condition
(\ref{D1}) is satisfied. However, there is no
 $\eta\ge0$ satisfying the matrix inequality in Assumption C.

\section{Conclusion and Further Research}
\label{sec:5} In this paper, we study a quadratic fractional
programming problem (P) over the intersection of an upper and a
lower level set of a quadratic function $g(x)$. In contrast to the
traditional Dinkelbach iterative method, we solve (P) by
establishing the equivalence between the parametric form
$(P)_{\lambda^*}$ and the related SDP formulations. Therefore,
computational efficiency for (P) is greatly improved over the
tedious and slow convergence of the repeated iterations.

The problem (P) is posed intensionally over the two-sided constraint
set in order to also shed some light on the old existing quadratic
programming with more than one quadratic constraint. However, our
study shows that the major difficulty of (P) lies in solving a
quadratic fractional minimization problem subject to a quadratic
{\it equality} constraint. The future research will be naturally to
obtain a stronger version of the extended S-Lemma and study its
geometric insights.

\end{document}